\documentclass{article}
\usepackage{amsthm,amsmath,amssymb,amsfonts}
\usepackage[english]{babel}
\usepackage{graphicx}
\usepackage[round]{natbib}
\usepackage{sidecap}
\usepackage{bbm}
\usepackage{enumitem}
\usepackage{color}
\usepackage[T1]{fontenc} 
\usepackage[utf8]{inputenc}
\usepackage{color}
\usepackage{ulem}

\def\title#1{{\Large\bf  \begin{center} #1 \vspace{0pt} \end{center}  } \smallskip}
\def\authors#1{{\bf \begin{center} #1 \vspace{0pt} \end{center} } \smallskip}
\def\institution#1{{\sl \begin{center} #1 \vspace{0pt} \end{center} } }

\def\keywords#1{\bigskip \par\noindent{\bf Keywords: }#1\par}
\def\AMS#1{\par\noindent{\bf AMS subject classifications: }#1\par}

\newtheorem{theorem}{Theorem}

\newtheorem{prop}[theorem]{Proposition}

\newtheorem{appxlem}{Lemma}[section]
\newtheorem{appxprop}[appxlem]{Proposition}

\newtheorem*{remark*}{Remark}

\newcommand{\R}{{\mathbb R}}
\newcommand{\E}{{\, \mathbb E}}
\renewcommand{\P}{{\mathbb P}}
\newcommand{\N}{{\mathbb N}}

\DeclareMathOperator{\Var}{Var}

\DeclareMathOperator{\trace}{trace}

\makeatletter
\newcommand{\vast}{\bBigg@{3}}
\newcommand{\Vast}{\bBigg@{5}}
\makeatother

\renewcommand{\textcolor}[2]{#2}
\renewcommand{\sout}[1]{}
\renewcommand{\xout}[1]{}
\renewcommand{\footnote}[1]{}

\begin{document}
\sloppy
\normalem

\title{Shrinkage estimators for prediction out-of-sample: \\
Conditional performance}

\authors{Nina Huber and Hannes Leeb}

\institution{Department of Statistics, University of Vienna}

\bigskip
\noindent {\Large\bf Abstract}
\medskip
 
\sout{In a linear model, we find that the
James-Stein estimator,} \textcolor{red}{We find that, in a linear model, the
James--Stein estimator}, which dominates the maximum-likelihood estimator
in terms of its in-sample prediction error, can perform poorly compared to
the maximum-likelihood estimator in out-of-sample prediction. We give a detailed analysis of this phenomenon and
discuss its implications\sout{ for statistical practice}. When evaluating the predictive performance of estimators, we
treat the regressor matrix in the training data as fixed, i.e., we condition on
the design variables.  Our findings contrast those obtained by \textcolor{red}{Baranchik (1973, Ann.\ Stat.\ 1:312--321)} and, more recently, by \textcolor{red}{Dicker (2012, arXiv:1102.2952)} in an unconditional performance evaluation.

\AMS{62M20, 62J07}

\section{Introduction}
\label{s1}

The problem of in-sample prediction, i.e., estimating the regression
function at the observed design points, is arguably among the most
extensively studied topics in regression analysis. 
But methods designed to perform well for\sout{ prediction in-sample} \textcolor{red}{in-sample prediction} need
not perform well for\sout{ prediction out-of-sample} \textcolor{red}{out-of-sample prediction,} i.e., for estimating the
regression function at a new point. We study the out-of-sample
predictive performance of the\sout{ James-Stein} \textcolor{red}{James--Stein} estimator, which dominates the maximum-likelihood
estimator in the in-sample scenario; see \cite{stein56} or the
comprehensive monograph \cite{judgebock78}. We focus on the\sout{ James-Stein} \textcolor{red}{James--Stein}
estimator because of its conceptual importance (and because it is
amenable to a detailed analytical analysis).\sout{ Indeed, } \textcolor{red}{The}\sout{ James-Stein} \textcolor{red}{James--Stein}
estimator is the first method that was found to dominate maximum-likelihood through shrinkage, a discovery that helped to spark the
development of many of the powerful estimation methods available today
that rely on some sort of shrinkage through, e.g., regularization,
model selection, or model averaging; see \cite{leebpoetscher08} for a
survey. In this paper, we find that the\sout{ James-Stein} \textcolor{red}{James--Stein} estimator can
perform poorly compared to the maximum-likelihood estimator in
out-of-sample prediction, and we analyze and explain this phenomenon.

Consider the Gaussian linear regression model
\begin{equation}\label{eq:model}
Y \quad=\quad X \beta + u,
\end{equation}
where $X$ is a fixed $n\times p$ matrix of rank $p$,
$\beta \in \R^p$, $n \geq p \geq 3$, and $u\sim N(0,\sigma^2 I_n)$.\sout{ For the sake of simplicity, we assume that $\sigma^2$ is known,
such that we can take $\sigma^2 = 1$ without loss of generality
in the following.}\footnote{\sout{
	Consideration of the known-variance case is sufficient to
	showcase the main points of the paper. 
	A detailed analysis of the unknown-variance case is 
	feasible, at the expense of additional lengthy computations,
	and hence will not be included here. 
}} \textcolor{red}{For simplicity, we focus on the known variance case and we assume that $\sigma^2 = 1$}. Given an estimator $\tilde{\beta}$ for $\beta$,
the corresponding in-sample prediction error, i.e.,
the mean squared error when estimating $X\beta$ by $X\tilde{\beta}$,
will be denoted by $\rho_1(\tilde{\beta},\beta,X)$ and is defined by
\begin{equation}\label{eq:ispe}
\rho_1(\tilde{\beta},\beta,X) \quad=\quad \frac{1}{n}
\E\Bigg[ (X\tilde{\beta} - X \beta)'(X\tilde{\beta} - X \beta)\Bigg]
\quad=\quad
\E\Bigg[ (\tilde{\beta}-\beta)' \frac{X'X}{n} (\tilde{\beta}-\beta)\Bigg].
\end{equation}
\sout{For prediction out-of-sample} \textcolor{red}{For out-of-sample prediction}, consider a new set of explanatory variables,
i.e., a $p$-vector $x_0$, that is independent of $Y$,
and hence also independent of  $\tilde{\beta}$, and that satisfies
$\E[x_0] = 0$ and $\E[ x_0 x_0'] = \Sigma$, where\sout{ $\Sigma>0$} \textcolor{red}{$\Sigma$ is positive definite and} is regarded
as a nuisance parameter. \textcolor{red}{(In case of fixed $x_0$, e.g., $x_0 = x_0^\ast \in \R^p$, we end up with the one-dimensional
	estimation target $x_0^\ast\mbox{}'\beta$, and it is\sout{ well-known} well known that
	the maximum-likelihood estimator for $x_0^\ast\mbox{}'\beta$ is 
	unique admissible minimax; see \cite{lehmanncasella98}.)} The out-of-sample prediction error is
the mean squared error when $x_0'\tilde{\beta}$  is used to 
predict $x_0' \beta$,
where now the mean is taken with respect to both $Y$ and $x_0$.\footnote{\sout{
	For a fixed value of $x_0$,
	e.g., $x_0= x_0^\ast \in \R^p$,  we end up with the one-dimensional
	estimation target $x_0^\ast\mbox{}'\beta$, and it is well-known that
	the maximum-likelihood estimator for $x_0^\ast\mbox{}'\beta$ is 
	unique admissible minimax;} see \cite{lehmanncasella98}.} This error will be denoted by $\rho_2(\tilde{\beta},\beta,X)$ and is defined by
\begin{equation}\label{eq:oospe}
\rho_2(\tilde{\beta},\beta,X) \quad=\quad 
\E\Bigg[ (x_0'\tilde{\beta} - x_0' \beta)^2 \Bigg]
\quad=\quad
\E\Bigg[ (\tilde{\beta}-\beta)' \Sigma (\tilde{\beta}-\beta)\Bigg].
\end{equation}
(Of course, the out-of-sample prediction error
$\rho_2(\tilde{\beta},\beta,X)$ also depends on the  matrix $\Sigma$, although this
dependence is not explicitly shown in our notation.) \textcolor{red}{We note that the existing results of \cite{baranchik73} and \cite{dicker12} consider prediction errors by assuming that $X$ is random and by taking expectations as in \eqref{eq:ispe} and \eqref{eq:oospe} also with respect to $X$. We, on the other hand, compute prediction errors by treating $X$ as fixed, i.e., we condition on the design.} The expressions on the far right-hand sides of \eqref{eq:ispe} and \eqref{eq:oospe}
differ in the matrices $X'X/n$ and $\Sigma$.
If $X'X/n$ is very close to $\Sigma$, then the in-sample prediction error
will be close to the out-of-sample prediction error.
But if $X'X/n$ is not very close to $\Sigma$, then
the in-sample predictive performance as measured by $\rho_1(\tilde{\beta},\beta,X)$
can be markedly different  from the out-of-sample predictive performance
as measured by $\rho_2(\tilde{\beta},\beta,X)$.

In this paper, we compare the maximum-likelihood estimator, the\sout{ James-Stein} \textcolor{red}{James--Stein} estimator, and related shrinkage-type estimators
by their performance as 
out-of-sample predictors. For in-sample prediction, i.e., in terms
of the risk $\rho_1(\cdot,\cdot,X)$, it is\sout{ well-known} \textcolor{red}{well known}
that the\sout{ James-Stein} \textcolor{red}{James--Stein} estimator dominates the maximum-likelihood estimator.
But for out-of-sample prediction, i.e., in terms of 
$\rho_2(\cdot,\cdot,X)$, we find that the\sout{ James-Stein} \textcolor{red}{James--Stein} estimator 
can perform quite poorly compared to the maximum-likelihood estimator;
see Figure~1, relation \eqref{eq:bad},
and also relation \eqref{t3.1} in Theorem~\ref{t3}.
(This finding contrasts a result of \cite{baranchik73} as 
discussed at the end of Section~\ref{s2} and after Theorem~\ref{t2} in
Section~\ref{s3}.)
But we also find that such disappointing worst-case performance
of the\sout{ James-Stein} \textcolor{red}{James--Stein} estimator is atypical, in a certain sense;
see relation~\eqref{t3.2} in Theorem~\ref{t3}.

For the case where
$\Sigma$ in \eqref{eq:oospe} is known, estimators that dominate the
maximum-likelihood estimator in terms of the risk \eqref{eq:oospe} are\sout{ well-known}
\textcolor{red}{well known}, and we refer to \cite{strawderman03} and the references given
therein. The case where $\Sigma$ is unknown but estimable is studied
by \cite{baranchik70}, \cite{bergerbock76}, \cite{bergeral77}, and
\cite{copas83}. For the challenging case where $\Sigma$ is unknown and
not estimable (in the sense that no further structural restrictions are
imposed on $\Sigma$ and that $p/n$ is large, a scenario that we study
in Section~\ref{s3}), we are not aware of further relevant existing results.

Explicit finite-sample formulae for the 
out-of-sample prediction errors of the estimators
in question are derived in Section 2. In Section 3, we present
approximations to the finite-sample quantities of interest;
our approximations become accurate as $n\to\infty$, 
uniformly in the underlying parameters.
Conclusions are drawn in Section~\ref{s4}, and the more technical
derivations are collected in the appendices.

\section{Explicit finite-sample results}
\label{s2}

Recall that
the maximum-likelihood estimator of $\beta$ is 
$\hat{\beta}_{ML} = (X'X)^{-1} X'Y\sim N(\beta, (X'X)^{-1})$.
As pointed out by \cite{stein56},\sout{ James-Stein-type} \textcolor{red}{James--Stein-type} shrinkage estimators
here correspond to estimators $\hat{\beta}(c)$ of $\beta$\sout{of} \textcolor{red}{with}\sout{ the form
$\hat{\beta}(c) = (1 - c p / (\hat{\beta}_{ML}' X' X \hat{\beta}_{ML}))\hat{\beta}_{ML}$} \textcolor{red}{$\hat{\beta}(c) = [1 - c p / (\hat{\beta}_{ML}' X' X \hat{\beta}_{ML})]
\hat{\beta}_{ML}$}, where $c \geq  0$ is a tuning-parameter
(cf. Appendix A). 
In particular,
the traditional\sout{ James-Stein} \textcolor{red}{James--Stein} estimator corresponds
to the estimator $\hat{\beta}(c)$ with $c=(p-2)/p$ 
and\sout{ we will denote this estimator} \textcolor{red}{will be denoted by} $\hat{\beta}_{JS}$\sout{ , i.e., $\hat{\beta}_{JS} = \hat{\beta}((p-2)/p)$};
in the following, $\hat{\beta}_{JS}$ will also be called
the\sout{ James-Stein} \textcolor{red}{James--Stein} estimator (of $\beta$). 

\begin{prop}
\label{propa1}
The in-sample prediction error and 
the out-of-sample prediction error of the\sout{ James-Stein-type} \textcolor{red}{James--Stein-type} shrinkage estimator
$\hat{\beta}{(c)}$ satisfy 
\begin{equation}\label{eq:ispe.js}
\rho_1(\hat{\beta}(c),\beta,X)\quad=\quad
\rho_1(\hat{\beta}_{ML},\beta,X) - \frac{1}{n} \left[ 2 cp  (p-2) - c^2 p^2
\right] \E\left[ \frac{1}{\hat{\beta}_{ML}' X' X \hat{\beta}_{ML}}\right]
\end{equation}
with $\rho_1(\hat{\beta}_{ML}, \beta, X)=p/n$, and
\begin{equation} \label{eq:oospe.js}
\begin{split}
\rho_2(\hat{\beta}(c),\beta,X)\quad=\quad
\rho_2(\hat{\beta}_{ML},\beta,X) & - 2 c p   \trace(\Sigma(X'X)^{-1})
\E\left[
\frac{1}{\hat{\beta}_{ML}' X' X \hat{\beta}_{ML}}
\right]
\\
& + ( c^2 p^2+ 4 c p
) \E\left[
\frac{\hat{\beta}_{ML}' \Sigma \hat{\beta}_{ML}
}{
(\hat{\beta}_{ML}' X' X \hat{\beta}_{ML})^2 }
\right]
\end{split}
\end{equation}
with $\rho_2(\hat{\beta}_{ML}, \beta, X)=\trace(\Sigma (X'X)^{-1})$,
respectively.
\end{prop}

The well-known formula on the
right-hand side of \eqref{eq:ispe.js} shows that the in-sample prediction error of 
$\hat{\beta}(c)$
equals the in-sample prediction error of 
$\hat{\beta}_{ML}$
minus the product of a positive expected value and a polynomial in $c$ 
and $p$. In particular, $\rho_1(\hat{\beta}(c),\beta,X)$ is
smaller than $\rho_1(\hat{\beta}_{ML},\beta,X)$ if $c$ satisfies
$0 < c  < 2(p-2)/p$ and is minimized for $c = (p-2)/p$,
which is the tuning-parameter used by $\hat{\beta}_{JS}$.
Moreover, it is easy to see 
that $\rho_1(\hat{\beta}(c),\beta,X)$ depends
on $\beta$ and $X$ only through $\beta' (X'X/n) \beta$. Unlike the
formula of $\rho_1(\hat{\beta}{(c)},\beta,X)$  in \eqref{eq:ispe.js}, display \eqref{eq:oospe.js} shows that
$\rho_2(\hat{\beta}{(c)},\beta,X)$ is obtained from 
$\rho_2(\hat{\beta}_{ML},\beta,X)$
by subtracting a positive term and then adding another  positive term, 
i.e, the second and the third term on the right-hand side of 
\eqref{eq:oospe.js},
that depend on $c$, on $p$, and on the unknown parameters in 
a \textcolor{red}{more} complicated fashion.
By further inspection, we find
that \eqref{eq:oospe.js} depends on $\beta$ and $X$ through
$\beta' (X'X/n)\beta$ and through $\beta'\Sigma\beta$
(which can be viewed as a kind of signal-to-noise ratio); for details,
see Proposition~\ref{propa2}.

\begin{center}
\begin{tabular}{cc}
\includegraphics[width=0.69 \textwidth]{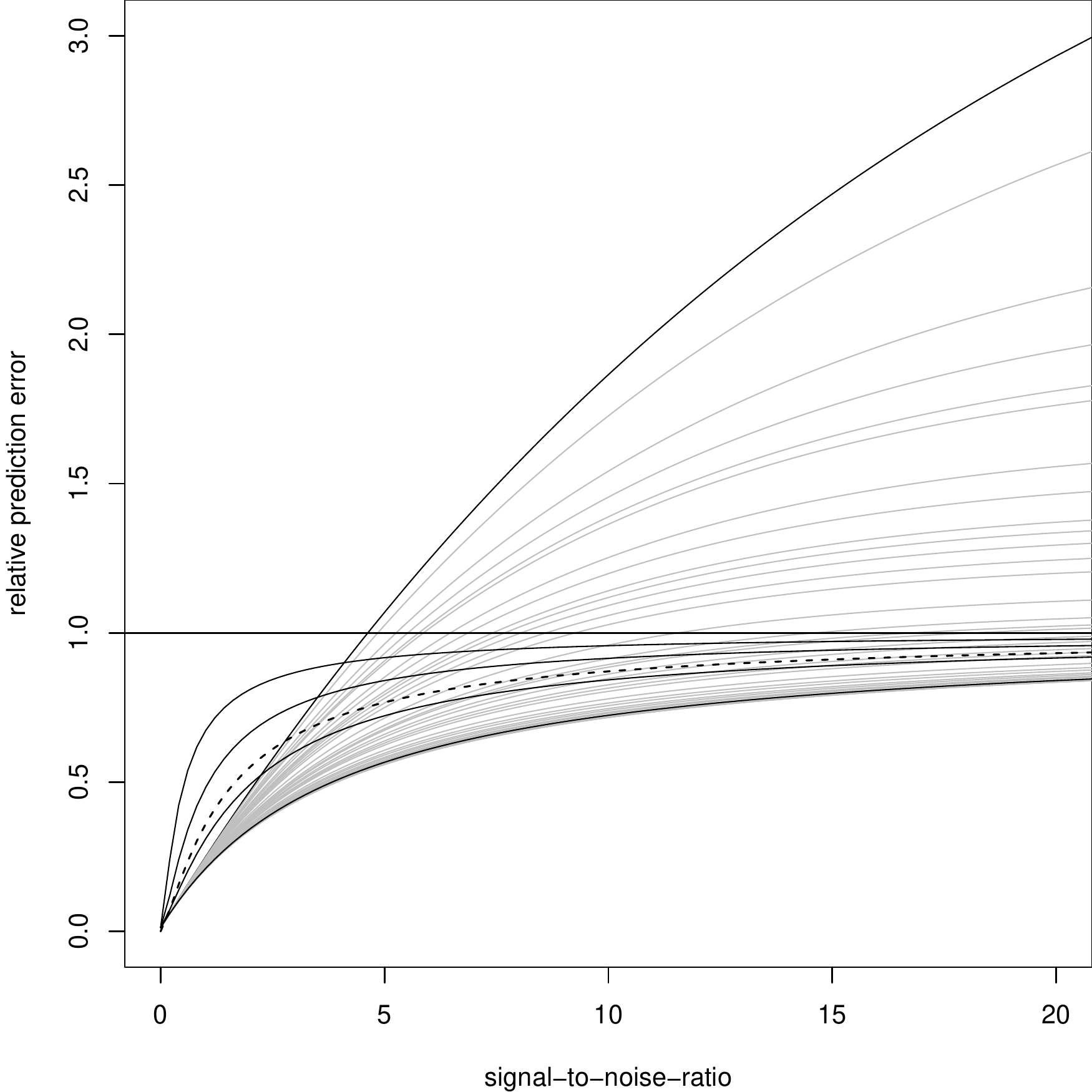}
& \begin{minipage}[b]{4.8 cm} \small{Figure 1: The solid curves in black and gray show\sout{ the} finite sample 
 	relative
     prediction errors\sout{ , i.e.,
          $\rho_2(\hat{\beta}_{JS},\beta,X)/\rho_2(\hat{\beta}_{ML},\beta,X)$, } as a function of\sout{ the signal-to-noise ratio} $\beta' \Sigma \beta$\sout{ for $\beta$ parallel to various eigenvectors
     of $X'X/n$}, and are explained in the next paragraph.
     The dashed curve shows the approximation
     to \textcolor{red}{the relative prediction error}\sout{ $\rho_2(\hat{\beta}_{JS},\beta,X)/\rho_2(\hat{\beta}_{ML},\beta,X)$,
     i.e., $r(\beta'\Sigma\beta, 1, p/n) / r(\beta'\Sigma \beta, 0, p/n)$,}
     that is obtained from Theorem~\ref{t1} in Section~\ref{s3}.
     The constant solid line at $1$ is for reference.} 
\end{minipage}
\end{tabular}
\end{center}

Figure~1 exemplifies the relative out-of-sample prediction error
of the\sout{ James-Stein} \textcolor{red}{James--Stein} estimator and of the maximum-likelihood estimator, i.e., 
$\rho_2(\hat{\beta}_{JS},\beta,X) / \rho_2(\hat{\beta}_{ML},\beta,X)$,
as a function of $\beta' \Sigma\beta$ 
for various configurations in parameter space.
For the figure, we selected a scenario 
where $X'X/n$ is not very close to $\Sigma$, so that  the
in-sample and the out-of-sample predictive performance of
estimators differ from each other (as noted in the discussion 
after \eqref{eq:oospe}). In particular, we took $n=200$,
$p=160$, $\Sigma= I_p$,  and $X$ was obtained by sampling i.i.d.\ standard normals.
The solid curves show 
$\rho_2(\hat{\beta}_{JS},\beta,X)/\rho_2(\hat{\beta}_{ML},\beta,X)$ 
as a function of $\beta'\Sigma\beta$, for $\beta$ parallel  to
various eigenvectors of $X'X/n$. Let $w_i$ be the eigenvector
corresponding to the eigenvalue $\nu_i$ of $X'X/n$, and
assume that $\nu_1 \leq \dots \leq \nu_{160}$.
Four solid black curves stay below $1$ and appear to be ordered;
starting from the top,
these correspond to $\beta$ parallel to $w_{160}$, $w_{120}$, $w_{80}$, 
and $w_{40}$.
The fifth solid black curve, that exceeds 1, corresponds to
$\beta$ parallel to $w_{1}$, i.e., the eigenvector of the smallest
eigenvalue. This curve attains a maximum of 4.27 at 75.12 (which
is off the chart) and
then recedes back towards $1$ as $\beta'\Sigma\beta \to \infty$.
The gray curves are obtained in the same way but for eigenvectors
corresponding to the remaining  smallest 25\% of eigenvalues.
The curves in Figure 1 differ dramatically depending on  whether they
correspond to small eigenvalues like $\nu_{1}$ on the one hand,
and moderate-to-large eigenvalues like $\nu_{40}$, $\nu_{80}$, 
$\nu_{120}$, and $\nu_{160}$ on the other.
Repeating these computations with $X$ replaced by a new independent
sample, we obtained essentially the same  results. And for other choices of
$p$ and $n$, we obtained results that are 
qualitatively similar, including maxima above $1$ corresponding to
small eigenvalues. This phenomenon becomes less pronounced as $p/n$
decreases, and it disappears completely for very small values of $p/n$.
The results in Section~\ref{s3} entail that
this is no surprise.

Figure~1 shows that the\sout{ James-Stein} \textcolor{red}{James--Stein} estimator 
no longer dominates the maximum-likelihood
estimator for out-of-sample prediction, in the sense that
\begin{equation}\label{eq:bad}
\rho_2(\hat{\beta}_{JS},\beta,X)
\quad > \quad
\rho_2(\hat{\beta}_{ML},\beta,X)
\end{equation}
for some $\beta \in \R^p$, if $X$ is the design matrix used to generate the figure.
(Indeed,  the left-hand side of the preceding display
exceeds the right-hand side by a factor of $4.27$ for appropriately
chosen $\beta$,
as noted in the preceding paragraph.)
This should be compared to the findings of \cite{baranchik73}: The results in that paper suggest that
 \begin{equation}\label{eq:good}
\E\left[\rho_2(\hat{\beta}_{JS},\beta,X)\right] \quad \leq \quad
\E\left[\rho_2(\hat{\beta}_{ML},\beta,X)\right]
\end{equation}
for each $\beta \in \R^p$, with strict inequality for some $\beta$, if $X$ is random with i.i.d.\ $N(0, \Sigma)$-distributed rows, and where the expectation in \eqref{eq:good} is taken with respect to $X$ (see also \cite{dicker12}). Comparing the preceding two displays, we see that for $X$ fixed, cf. \eqref{eq:bad}, $\hat{\beta}_{JS}$ can perform poorly for some $\beta \in \R^p$. But on average with respect to $X$, as considered in \cite{baranchik73} and \cite{dicker12}, the relation in \eqref{eq:good} suggests that $\hat{\beta}_{JS}$ performs well, irrespective of $\beta$. The performance of $\hat{\beta}_{JS}$ hence depends crucially on whether we condition on the design as in \eqref{eq:bad} or average with respect to the design distribution as in \eqref{eq:good}. We give a more detailed analysis and explanation of this phenomenon
in the next section. Also note that
the phenomenon in \eqref{eq:bad} and \eqref{eq:good} is related to
the ancillarity paradox of \cite{brown90}.

\section{Asymptotic approximations}
\label{s3}

In this section, we provide approximations to  quantities like 
$\rho_2(\hat{\beta}_{JS},\beta,X)$ and
$\sup_{\beta} \rho_2(\hat{\beta}_{JS},\beta,X)$
for `typical' design matrices $X$.
Here, `typical' means `in probability' when  the explanatory variables 
in the training period,
i.e., the rows of $X$, are taken as realizations from the same 
distribution as those in the prediction period
(i.e., $x_0$).  Our approximations are uniform in the unknown parameters
and become accurate as $n\to\infty$,
where the dimension of the model considered at sample size $n$, i.e.,
$p$, is allowed to depend on sample size.\footnote{\sout{
	Our asymptotic approximations to the out-of-sample prediction error
	are most useful in the high-dimensional case
	where $p/n$ is bounded away from zero.
	In the remaining case, i.e.,  when $p/n \to 0$, 
	the out-of-sample prediction
	error typically converges to zero, 
	such that estimators should be compared in terms of relative
	errors.  
	Relative out-of-sample prediction errors will be studied elsewhere.}
} Note that quantities like $\rho_2(\hat{\beta}{(c)}, \beta, X)$ and
$\rho_2(\hat{\beta}_{ML}, \beta, X)$ now become random variables
through their dependence on $X$. We emphasize that our evaluation of performance
is always taken conditional on $X$, and that the
random design is used only to describe the behavior for `typical' 
design matrices $X$.
We also stress, with $X$ random, that the expectations in 
\eqref{eq:ispe.js} and
\eqref{eq:oospe.js} are now to be understood as conditional on $X$.

\sout{More formally, the following assumptions will be maintained throughout
this section:} For each $n$ and $p$ under consideration ($n\geq p\geq 3$), 
we assume that the 
model~\eqref{eq:model} holds; that $x_0$ and $X$ are independent of
the error $u$ in \eqref{eq:model}; and that the rows of $X$ and also $x_0$ are i.i.d.\ 
with mean zero and\sout{ covariance matrix $\Sigma>0$} \textcolor{red}{positive definite covariance matrix $\Sigma$}. In addition, we
assume that $X$ can be written as $X = V \Sigma^{1/2}$, where
$\Sigma^{1/2}$ denotes a symmetric square root of $\Sigma$ and
where $V$ is the $n\times p$ matrix obtained by taking the upper left block of 
a double \textcolor{red}{infinite} array $(V_{i,j})_{i\geq 1,j\geq 1}$ 
of i.i.d.\ random variables that
have mean zero,  variance one, and a finite fourth moment\footnote{\sout{
	This assumption is widely used in the literature on large-dimensional
	random matrices; cf. }\cite{baisilverstein10}.\sout{  
	We impose this assumption here because
	we rely on results from this literature.}
} \textcolor{red}{(cf. \cite{baisilverstein10})}.
Finally, we also assume that the (marginal) distribution of the
$V_{i,j}$'s is absolutely continuous with respect to Lebesgue
measure.\footnote{\sout{This assumption serves only to shorten some of the
proofs and can be dropped altogether. The asymptotic results in this
section continue to hold without this assumption if
$\rho_2(\hat{\beta}(c), \beta, X)$ is defined as, say, 0 on the event
where $X'X$ or, equivalently, $V'V$, is degenerate, because the probability of that event goes
to zero as $n \to \infty$ whenever $p/n$ is bounded away from 1;}
cf. Lemma~\ref{mp2}.} \textcolor{red}{This assumption could be dropped altogether, but at the expense of longer and technically more involved proofs.}
Under these assumptions, we note that\sout{ $X'X>0$} \textcolor{red}{$X'X$ is invertible} almost surely. We set
$ \rho_2(\hat{\beta}(c),\beta,X) =  0$
on the probability-zero event where $X'X$ is degenerate. Otherwise,
the random variable $\rho_2(\hat{\beta}(c),\beta,X)$
is defined by the expression on the right-hand side of
\eqref{eq:oospe.js}, where the expected values are to be understood as
conditional on $X$. These conventions also cover
$\rho_2(\hat{\beta}_{ML},\beta,X)$
and $\rho_2(\hat{\beta}_{JS},\beta,X)$ in view of
$\hat{\beta}_{ML} = \hat{\beta}(0)$ and 
$\hat{\beta}_{JS} = \hat{\beta}((p-2)/p)$. The following two results
provide simple asymptotic approximations to quantities like
$\rho_2(\hat{\beta}_{JS}, \beta, X)$ as well as $\sup_{\beta \in \R^p}
\rho_2(\hat{\beta}_{JS}, \beta, X)$.

\begin{theorem}\label{t1}
Assume that $n\to\infty$, and that $p=p(n)$ is such that 
$p/n \to t \in [0,1)$. Moreover, for  each $p$, let $\beta$ and
$\Sigma$ be a $p$-vector and a positive definite $p\times p$ matrix,
respectively, so that
$\beta'\Sigma\beta \to \delta^2 \in [0,\infty]$ as $n\to\infty$. Then
$\rho_2(\hat{\beta}_{ML}, \beta,X) \to t/(1-t)$ and
$\rho_2(\hat{\beta}{(c)}, \beta,X) \to r(\delta^2,c,t)$ in probability,
where
\begin{equation}\nonumber
r(\delta^2,c,t) \quad=\quad
\frac{t}{1-t} 
\left(
1 - c \frac{t}{t+\delta^2}
\right)^2
\;+\;
c^2 \frac{t^2 \delta^2}{(t + \delta^2)^2}
\end{equation}
if $\delta^2$  is finite, and where $r(\infty,c,t) = t/(1-t)$ otherwise
(expressions like $t/(t+\delta^2)$ are to be interpreted
as zero if $t$ and $\delta^2$ are both equal to zero).\sout{ In fact, convergence} \textcolor{red}{Convergence} of $\rho_2(\hat{\beta}{(c)},\beta,X)$ 
is uniform in $c$ over compact sets, in the sense that\sout{ $\sup_{0\leq c\leq C} | \rho_2(\hat{\beta}{(c)},\beta,X) - 
r(\delta^2,c,t)|$ converges to zero in probability} \textcolor{red}{$\sup_{0\leq c\leq C} | \rho_2(\hat{\beta}{(c)},\beta,X) - 
r(\delta^2,c,t)|=o_P(1)$} for any $C>0$.
Moreover, these statements continue to hold if $\rho_2(\hat{\beta}{(c)},\beta,X)$ is replaced by the estimator
$r(\hat{\delta}^2,c,p/n)$ where $\hat{\delta}^2 = \max\{ Y'Y/n-1,0\}$.
\end{theorem}

\begin{theorem}\label{t2}
Assume that $n\to\infty$, and that $p=p(n)$ is such that $p/n\to t\in [0,1)$.
Moreover, for each $p$, let $\Sigma$ be a positive definite $p\times p$ 
matrix. Then 
$\sup_{\beta\in \R^p} \rho_2(\hat{\beta}{(c)},\beta,X) \to
\sup_{\delta^2\geq 0} R( \delta^2, c,t)$ in probability, where
\begin{equation}\nonumber
R(\delta^2, c,t) \quad=\quad
\frac{t}{1-t} 
\left(
1 - c \frac{t}{t+\delta^2}
\right)^2
\;+\;
c^2 \frac{t^2 \delta^2}{(1-\sqrt{t})^2 (t + \delta^2)^2}
\end{equation}
(again, expressions like $t/(t+\delta^2)$ are to be interpreted
as zero if $t$ and $\delta^2$ are both equal to zero).\sout{ In fact, convergence} \textcolor{red}{Convergence} is uniform in $c$ over compact sets,
so that\sout{ $\sup_{0\leq c\leq C} | \rho_2(\hat{\beta}{(c)},\beta,X) - 
r(\delta^2,c,t)|$ converges to zero in probability} \textcolor{red}{
$\sup_{0\leq c \leq C} | 
\sup_{\beta\in \R^p} \rho_2(\hat{\beta}{(c)},\beta,X) -
\sup_{\delta^2\geq 0} R( \delta^2, c,t) |=o_P(1)$} for any  $C>0$.
\end{theorem}

\begin{remark*}\normalfont
The quantities $r(\delta^2,c,t)$ and $R(\delta^2,c,t)$, as defined in\sout{ theorems}
\textcolor{red}{Theorems}~\ref{t1} and~\ref{t2}, respectively, differ by a factor of
$(1-\sqrt{t})^2$ in the denominator of the last term. 
The proofs reveal that this factor is caused by the fact that
$\beta'(X'X/n) \beta$ can differ considerably from its expectation\sout{,
i.e., from $\beta'\Sigma \beta$,} if $p/n$ is not small, because of the gap
between the smallest eigenvalue of a large-dimensional random Wishart matrix
(or Gram matrix) and the smallest eigenvalue of its expectation, a well-known phenomenon
in the theory of random matrices; see, for
example,~\cite{baisilverstein10}. In particular, in the setting of
Theorem~\ref{t1} with $\delta^2=1$, $\beta'(X'X/n) \beta$
converges to\sout{ 1} \textcolor{red}{one} in probability, but $\inf_{b:b'\Sigma b=1} b'(X'X/n)b$
converges to $(1-\sqrt{t})^2$.
\end{remark*}

The approximations to $\rho_2(\hat{\beta}(c),\beta,X)$ and 
$\rho_2(\hat{\beta}_{ML},\beta,X)$ provided by Theorem~\ref{t1}, i.e.,
$r(\delta^2,c,t)$ and $t/(1-t)$, respectively, are such that 
$r(\delta^2,c,t) < t/(1-t)$ whenever $t>0$, provided only that
$0<c\leq 2$ and $\delta^2<\infty$; cf. the dashed line in Figure~1.
The results of \cite{dicker12}\sout{ (who requires that $X$ is Gaussian and $p\to\infty$)} suggest approximations to
$\E[\rho_2(\hat{\beta}(c),\beta,X)]$ and 
$\E[\rho_2(\hat{\beta}_{ML},\beta,X)]$
that coincide with $r(\delta^2,c,t)$ and $t/(1-t)$, respectively.

Theorems~\ref{t1} and~\ref{t2} together with the attending remark
also provide us with a more precise description of the phenomenon in
\eqref{eq:bad} and \eqref{eq:good}, and with a better understanding of the
underlying cause. More formally, we have the following result.

\begin{theorem}\label{t3}
Assume that $n\to\infty$, and that $p=p(n)$ is such that $p/n\to t\in [0,1)$.
Moreover, for each $p$, let $\Sigma$ be a positive definite $p\times p$ 
matrix. If $t > 1/9$, then  the out-of-sample prediction errors of
the\sout{ James-Stein} \textcolor{red}{James--Stein} estimator and of the maximum-likelihood estimator 
are such that
\begin{equation}\label{t3.1}
\P\left(
\sup_{\beta \in \R^p}  
\rho_2( \hat{\beta}_{JS}, \beta, X) - 
	\rho_2(\hat{\beta}_{ML}, \beta, X)
	\,> \, \epsilon\right) \quad \xout{\stackrel{n\to\infty}{\longrightarrow}}
	\textcolor{red}{\to}  \quad 1
\end{equation}
for some $\epsilon>0$ (that is given explicitly in the 
proof).
And if $t \leq 1/9$, then  the expression on the left-hand side of
\eqref{t3.1} converges to zero as $n\to\infty$ for each
$\epsilon>0$.
Finally, irrespective of the value of $t\in [0,1)$, we have
\begin{equation}\label{t3.2}
\sup_{\beta \in \R^p}  
\P\left(
\rho_2( \hat{\beta}_{JS}, \beta, X) - 
	\rho_2(\hat{\beta}_{ML}, \beta, X)
	\,> \, \epsilon\right) \quad \xout{\stackrel{n\to\infty}{\longrightarrow}}
	\textcolor{red}{\to}  \quad 0
\end{equation}
for each $\epsilon>0$.

More generally, consider tuning-parameters $c_n\geq 0$ that converge
to a limit $c \in [0,\infty)$ as $n\to\infty$, and consider
the expression on the left-hand side of
\eqref{t3.1} with $\hat{\beta}(c_n)$ replacing $\hat{\beta}_{JS}$.
The resulting expression converges to one as $n\to\infty$ for some $\epsilon>0$
in the case where
$0 \leq  c \leq 2$ and\sout{ $t > ((c-2)/(c+2))^2$} \textcolor{red}{$t > [(c-2)/(c+2)]^2$}, and in the case where
$c> 2$ and $t>0$.
In all other cases,
the resulting expression converges to zero for each $\epsilon>0$.
And \eqref{t3.2} holds for each $\epsilon >0$ with
$\hat{\beta}(c_n)$ replacing $\hat{\beta}_{JS}$
if $c \leq 2$, irrespective of $t$.
\end{theorem}

\begin{remark*}\normalfont
In the setting of Theorem~\ref{t3}, it is easy to see
that relation \eqref{t3.1} holds
uniformly over all pairs of $n$ and $p$ subject to $1/9+\delta\leq p/n
\leq 1-\delta$;
in addition, \eqref{t3.2} holds uniformly over all
such pairs with $p/n \leq 1-\delta$,
for each $\delta>0$, subject to $n\geq p \geq 3$
(and also uniformly over all\sout{ $p\times p$ matrices $\Sigma>0$} \textcolor{red}{positive definite $p\times p$ matrices $\Sigma$}).
Similar statements also apply with $\hat{\beta}(c_n)$ replacing
$\hat{\beta}_{JS}$, mutatis mutandis.
\end{remark*}

Through relations \eqref{t3.1} and \eqref{t3.2},
Theorem~\ref{t3} provides two complementing views on
the worst-case performance of\sout{ James-Stein-type} \textcolor{red}{James--Stein-type} shrinkage estimators. If the expression on the left-hand side of  \eqref{t3.1} is large, then 
the\sout{ James-Stein} \textcolor{red}{James--Stein} estimator $\hat{\beta}_{JS}$ 
is typically outperformed, from a worst-case perspective,  by
the maximum-likelihood estimator $\hat{\beta}_{ML}$ 
(whose out-of-sample prediction error is constant in $\beta$).
Here, `typically' means that for most realizations of the design matrix
$X$ there is a parameter $\beta$ for which $\hat{\beta}_{JS}$ 
performs worse than $\hat{\beta}_{ML}$.
By Theorem~\ref{t3}, we see that this occurs with probability approaching
one in the statistically challenging case where $t>1/9$
(while this occurs with probability approaching zero in the case where
$t\leq 1/9$\sout{, e.g., in cases where  the sample size exceeds the number
of regressors in the model by a factor of $9$ or more}).
On the other hand, if the expression on the left-hand side of \eqref{t3.2}
is small, then with high probability $X$ is such that
$\hat{\beta}_{JS}$ outperforms $\hat{\beta}_{ML}$, uniformly in $\beta$. In the setting of Theorem~\ref{t3}, 
the left hand-side of \eqref{t3.2} always converges to zero.

If $\rho_2(\cdot, \cdot, X)$ is used as a risk-function, 
then \eqref{t3.1} entails that $\hat{\beta}_{JS}$
is outperformed by $\hat{\beta}_{ML}$ in terms of worst-case risk,
for most realizations of the design matrix $X$.\sout{ Note that this} \textcolor{red}{This} is a  worst-case perspective, as is often adopted 
in frequentist statistical analyses.
But the most unfavorable parameter $\beta$, for which
$\rho_2(\hat{\beta}_{JS},\beta,X)$ is maximized, heavily depends on $X$.
In particular, the relation in \eqref{t3.2} entails,
for any fixed parameter $\beta$, that the probability, that $X$ is such that
$\beta$ is unfavorable, is small.

\section{\sout{Conclusions} \textcolor{red}{Discussion}}
\label{s4}

We have derived explicit finite sample formulae for the out-of-sample
prediction error of the\sout{ James-Stein} \textcolor{red}{James--Stein} estimator and of related\sout{ James-Stein-type} \textcolor{red}{James--Stein-type} shrinkage estimators in a linear regression model
with Gaussian errors and fixed design.
In an example with a particular design matrix $X$,
we have found that the\sout{ James-Stein} \textcolor{red}{James--Stein} estimator no longer dominates the 
maximum-likelihood estimator.
We have shown that this phenomenon generally occurs for most design
matrices $X$ if the ratio  of the number of explanatory variables in the model
($p$) and the sample size ($n$)\sout{, i.e, $p/n$,} exceeds $1/9$,
in the sense of statement \eqref{t3.1} of Theorem~\ref{t3}.
At the same time, we have also shown that the\sout{ James-Stein} \textcolor{red}{James--Stein} estimator
outperforms the maximum-likelihood estimator for most design matrices
$X$, uniformly in the underlying parameters, in the sense
of statement \eqref{t3.2} of Theorem~\ref{t3}.
Our findings suggest that the\sout{ James-Stein} \textcolor{red}{James--Stein} estimator can perform
poorly for prediction out-of-sample from a frequentist worst-case
perspective. But our findings also suggest, in the setting considered
here, that the worst-case
performance does not properly reflect the performance in the typical case,
and that the\sout{ James-Stein} \textcolor{red}{James--Stein} estimator performs favorably compared to 
maximum-likelihood in the typical case, uniformly in the underlying parameters.

The phenomenon, that the\sout{ James-Stein} \textcolor{red}{James--Stein} estimator dominates the maximum-likelihood
estimator for in-sample prediction but can fail to do so for
prediction out-of-sample, is linked to the fact that the eigenvalues
of $X'X/n$ can differ from the eigenvalues of $\Sigma$; see relations
\eqref{eq:ispe} and \eqref{eq:oospe}, as well as the remark following
Theorem~\ref{t2}. We therefore expect that other estimators,
that are designed to perform well for in-sample prediction,
can exhibit similar phenomena when used for prediction out-of-sample.
This includes various shrinkage estimators like 
estimators based on model selection, penalized maximum-likelihood,
and other forms of regularization. Although beyond the scope of this
paper, it would be particularly interesting to study\sout{ the so-called} bridge estimators
\citep{frankfriedman93},\sout{ and,} in particular, the LASSO
\citep{tibshirani96} \textcolor{red}{and ridge regression \citep{hoerlkennard70}}, as well as the Dantzig selector \citep{candestao07} with
regards to their performance as predictors out-of-sample when $p/n$
is not small.

\begin{appendix}

\section{Technical details for Section~\ref{s2}}
\label{appendixA}

Recall that $\hat{\beta}_{ML} \sim N(\beta, (X'X)^{-1})$,
and define $Z$ and $\zeta$ as $Z = (X'X)^{1/2}\hat{\beta}_{ML}$
and $\zeta = (X'X)^{1/2} \beta$, respectively, where $(X'X)^{1/2}$ denotes a
symmetric square root of $X'X$. Note that we have $Z \sim N(\zeta,I_p)$,
that the unknown parameter $\beta\in \R^p$ corresponds to the unknown
parameter $\zeta\in \R^p$ via $\beta = (X'X)^{-1/2} \zeta$, and that
any estimator $\tilde{\beta}$ for $\beta$ corresponds to an estimator
$\tilde{\zeta}$ for $\zeta$ via the relation 
$\tilde{\beta} = (X'X)^{-1/2}\tilde{\zeta}$, and vice versa.
In the Gaussian location model $Z \sim N(\zeta,I_p)$ with
$\zeta\in \R^p$, $p\geq 3$, 
consider the\sout{ James-Stein-type} \textcolor{red}{James--Stein-type} shrinkage estimator
$\hat{\zeta}{(c)} = (1-c p /Z'Z)Z$ with 
tuning parameter $c\geq 0$.
The estimator $\hat{\zeta}{(c)}$ for $\zeta$  corresponds to the
estimator $(X'X)^{-1/2}\hat{\zeta}{(c)}$ for $\beta$, and
it is elementary to verify that
$(X'X)^{-1/2}\hat{\zeta}{(c)}$ equals $\hat{\beta}{(c)}$.

\begin{proof} [\bf Proof of Proposition~\ref{propa1}:]
Define $Z$, $\zeta$ and $\hat{\zeta}{(c)}$ as in the preceding
paragraph. For the in-sample prediction error of the maximum-likelihood estimator, note that $\rho_1(\hat{\beta}_{ML}, \beta,X) =
(1/n) \E[(Z - \zeta)'(Z - \zeta)]=p/n$. For relation~\eqref{eq:ispe.js}, recall that $\hat{\zeta}{(c)}$ satisfies
$\E[(\hat{\zeta}{(c)} - \zeta)'(\hat{\zeta}{(c)} - \zeta)]
=p - [2 c p (p-2) - c^2 p^2] \E[ 1/Z'Z]$; cf. \cite{jamesstein61}.\sout{ It is now elementary to} \textcolor{red}{To} verify that this equality is equivalent to\sout{ the equality in} \eqref{eq:ispe.js}\sout{. [First}\textcolor{red}{, first} note that the expression on the left-hand side
of this equality satisfies
$\E[(\hat{\zeta}{(c)}-\zeta)'(\hat{\zeta}{(c)}-\zeta)] 
= n \rho_1(\hat{\beta}{(c)}, \beta, X)$ (by
plugging the definitions of $\hat{\zeta}{(c)}$, $Z$ and $\zeta$ 
into the left-hand side, and by recalling the definitions of
$\hat{\beta}{(c)}$ and $\rho_1(\hat{\beta}{(c)},\beta, X)$).
And, in a similar fashion, the right-hand side of that equality
is equal to the expression on the right-hand side of \eqref{eq:ispe.js},
multiplied by $n$.\sout{]}

For\sout{ relation} \eqref{eq:oospe.js}, we first use\sout{ the definition
of the out-of-sample prediction error, cf.} \eqref{eq:oospe}
together with the definition of $\hat{\beta}{(c)}$ to obtain that 
\begin{align}
\label{tmp1}
\rho_2(\hat{\beta}{(c)},\beta,X)\quad=\quad
\rho_2(\hat{\beta}_{ML},\beta,X) 
\;-\; 2 c p \E\left[
\frac{\hat{\beta}_{ML}' \Sigma(\hat{\beta}_{ML}-\beta)}{
\hat{\beta}_{ML}'X'X \hat{\beta}_{ML}}
\right]
\;+\; c^2 p^2 \E\left[
\frac{
	\hat{\beta}_{ML}'\Sigma \hat{\beta}_{ML}
}{
( \hat{\beta}_{ML}'X'X \hat{\beta}_{ML})^2
}
\right].
\end{align}
The expected value in the second term on the right-hand side of~\eqref{tmp1}
can be written as
\begin{align}\label{tmp2}
\E \left[
\frac{Z' T (Z - \zeta) }{ Z'Z }
\right]
\quad=\quad
\sum_{i=1}^p
\E\left[
	\frac{
		Z_i T_{i,i} (Z_i - \zeta_i)
	}{
	\sum_{k=1}^p Z_k^2
	}
\right]
\;+\;
\sum_{\stackrel{i,j =1}{i\neq j}}^p
\E\left[
	\frac{
		Z_j T_{j,i} (Z_i - \zeta_i)
	}{
	\sum_{k=1}^p Z_k^2
	}
\right],
\end{align}
where $Z$ and $\zeta$ have been defined earlier,  and where\sout{ we write $T = (T_{i,j})_{i,j=1}^p$ as shorthand for the matrix} \textcolor{red}{$T = (X'X)^{-1/2} \Sigma (X'X)^{-1/2}$.}
For each of the
terms in the first sum on the right-hand side of \eqref{tmp2},  we obtain that
\begin{align*}
\E\left[
	\frac{
		Z_i T_{i,i} (Z_i - \zeta_i)
	}{
	\sum_{k=1}^p Z_k^2
	}
\right]
&\quad=\quad
\E \left[
	\frac{
		T_{i,i} 
	}{
	\sum_{k=1}^p Z_k^2
	}
\right]
\;-\;
2 \E \left[
	\frac{
		T_{i,i}  Z_i^2
	}{
	(\sum_{k=1}^p Z_k^2)^2
	}
\right]
\end{align*}
upon using Stein's Lemma conditional on the $Z_k$'s with $k\neq i$.
And each of the terms in the second (double) sum on the right-hand side
of~\eqref{tmp2} can be written as
\begin{align*}
\E\left[
	\frac{
		Z_j T_{j,i} (Z_i - \zeta_i)
	}{
	\sum_{k=1}^p Z_k^2
	}
\right]
\quad=\quad 
-\; 2 \E \left[
	\frac{
		T_{j,i}  Z_i Z_j
	}{
	(\sum_{k=1}^p Z_k^2)^2
	}
\right],
\end{align*}
again by using Stein's Lemma conditional on the $Z_k$'s with $k\neq i$.
Using the preceding two equalities, we can write the right-hand side
of \eqref{tmp2} as
$$
\trace(T)\E\left[ \frac{1 }{ Z'Z} \right]
\;-\;
2 \E\left[ \frac{Z'T Z}{ (Z'Z)^2} \right]
\quad=\quad
\trace(\Sigma (X'X)^{-1}) \E\left[
\frac{1}{ 
\hat{\beta}_{ML}'X'X\hat{\beta}_{ML}}
\right]
\;-\; 
2 \E \left[
\frac{\hat{\beta}_{ML}' \Sigma\hat{\beta}_{ML}} {
(\hat{\beta}_{ML}'X'X\hat{\beta}_{ML})^2}
\right],
$$
where the equality is obtained by using the definitions of $T$ and $Z$\sout{, together with the fact that $\trace( A B) = \trace(B A)$
for matrices $A$ and $B$ of appropriate dimensions}. Recall that \eqref{tmp2} or, equivalently, the right-hand side of the preceding display, is equal to the expected value in the second term on the right-hand side of \eqref{tmp1}. Plugging this into \eqref{tmp1} and simplifying, we obtain
\eqref{eq:oospe.js}. Furthermore,\sout{ note that} $\rho_2(\hat{\beta}_{ML}, \beta, X)= \trace( \Sigma
\E[(\hat{\beta}_{ML}-\beta)(\hat{\beta}_{ML}-\beta)']=\trace(\Sigma (X'X)^{-1})$.
\end{proof}

The next result shows how\sout{  the expressions in} \eqref{eq:ispe.js} and \eqref{eq:oospe.js} 
depend on $\beta$ and on $X$; moreover, that result can also be used
to compute these expressions 
numerically.

\begin{appxprop}\label{propa2} \sout{In the setting of Section~\ref{s2}, }
\textcolor{red}{Assume that the linear model in \eqref{eq:model} holds and that $\Sigma$ is a positive definite $p \times p$ matrix.} \textcolor{red}{Then} the expected values in 
\eqref{eq:ispe.js} and \eqref{eq:oospe.js} 
can be written as
\begin{align*}
\E \left[
\frac{1}{\hat{\beta}_{ML}' X'X \hat{\beta}_{ML}}
\right]
& \quad=\quad 
\E \left[
	\frac{1}{\chi^2_p(\beta' X'X \beta)}
\right],
\\
\E \left[
\frac{\hat{\beta}_{ML}' \Sigma \hat{\beta}_{ML}
}{
\left(\hat{\beta}_{ML}' X'X \hat{\beta}_{ML}\right)^2
}
\right]
& \quad=\quad 
\E\left[
\frac{\trace(\Sigma(X'X)^{-1}) 
}{\left( \chi^2_{p+2}(\beta' X'X\beta)\right)^2}
\right]
\;+\;
\E\left[
\frac{ \beta'\Sigma\beta 
}{\left( \chi^2_{p+4}(\beta' X'X\beta)\right)^2}
\right],
\end{align*}
where, for each $k\geq 1$ and each $\lambda\geq 0$,
the symbol $\chi^2_k(\lambda)$ denotes
a random variable that is chi-square distributed with $k$ degrees
of freedom and non-centrality parameter $\lambda$.
\end{appxprop}

\begin{proof}
The first equality follows upon recalling that
$X \hat{\beta}_{ML}$ follows a Gaussian distribution with 
mean $X \beta$ and covariance matrix $X (X'X)^{-1} X'$.
The second equality is obtained by applying Corollary 2 from the
appendix of \cite{Bock75}.
\end{proof}

\section{Technical details for Section~\ref{s3}}

We start with\sout{ a couple of} \textcolor{red}{some} auxiliary results, and then\sout{ proceed to} prove the results from Section~\ref{s3}.

\begin{appxlem}
\label{mp1}
For $t\geq 0$, $a=(1-\sqrt{t})^2$ and $b=(1+\sqrt{t})^2$, we have
\begin{equation*}
\int_a^b \frac{1}{x^2} \sqrt{(b-x)(x-a)} \text{d} x = 2 \pi 
\frac{\min\{1,t\}}{|1-t|},
\end{equation*}
where the expression on the right-hand side of the preceding display
is to be interpreted as $\infty$ in case $t=1$.
\end{appxlem}

\begin{proof}
The case $t=0$ is trivial. In the case where $t >0$, 
the integral of interest can be written as
\begin{align*}
\int_a^b \frac{1}{x^2} \sqrt{(b-x)(x-a)} \textcolor{red}{\text{d}} x  &\quad =\quad
2(b-a)^2\int_0^{\infty} \frac{z^2}{(az^2+b)^2(z^2+1)} \textcolor{red}{\text{d}}z \\
&\quad =\quad 2(b-a)^2 \int_0^{\infty} 
\textcolor{red}{\left[ \frac{a}{(b-a)^2}\frac{1}{(az^2+b)} +
\frac{b}{b-a}\frac{1}{(az^2+b)^2} -
\frac{1}{(b-a)^2}\frac{1}{z^2+1} \right] \text{d}}z,
\end{align*}
where the first equality is obtained by the substitution
$z = \sqrt{(b-x)/(x-a)}$, and the second equality follows from a 
partial fraction decomposition.
Recall that $\int_0^w (u^2+1)^{-1} \text{d} u = \arctan(w)$
and that $2 k \int_0^w (u^2+1)^{-(1+k)} \text{d} u = 
w (w^2+1)^{-k}  + (2k-1) \int_0^w (u^2+1)^{-k} \text{d} u $
whenever $k >0$.
With this, and for the case where $t\neq 1$,
it\sout{ is elementary to verify} \textcolor{red}{follows} that the
integral on the far right-hand side of the preceding display 
reduces to
$$
2 (b-a)^2 \left[
\frac{a}{(b-a)^2} \frac{\pi}{2 \sqrt{a} \sqrt{b}}
+ \frac{b}{b-a} \frac{\pi}{4 \sqrt{a}\sqrt{b^3} } 
-\frac{1}{(b-a)^2} \frac{\pi}{2}
\right]
\quad=\quad
\frac{\pi}{2\sqrt{a}\sqrt{b}} (\sqrt{b}-\sqrt{a})^2.
$$
Now note that $\sqrt{a} = 1-\sqrt{t}$ in case $t<1$,
while $\sqrt{a} = \sqrt{t}-1$ in case $t>1$.\sout{ With this, it is
elementary to verify} \textcolor{red}{This implies} that the expression on the far right-hand 
side of the preceding display reduces to $2 \pi \min\{1,t\} / |1-t|$, 
as required.
And in case $t=1$, we have $a=0$ and $b=4$, so that
the integrand on the far right-hand side of the second-to-last display 
reduces to\sout{ $( 1 - 1/(z^2+1))/ 16$} \textcolor{red}{$[ 1 - 1/(z^2+1)]/ 16$} and hence integrates to $\infty$.
\end{proof}

\begin{appxlem}
\label{mp2}
For the $n\times p$ matrix $V$ as in Section~\ref{s3}, write
$\lambda_1\geq \lambda_2\geq \dots \geq  \lambda_p$ for the ordered eigenvalues
of $V'V$.  If $n$ and $p$ are such that $n\to\infty$  and $p/n \to t$
for some $t\in [0,1]$, 
then
$$
\sum_{i=1}^p \frac{1}{\lambda_i} 
\quad\stackrel{\mbox{a.s.}}{\longrightarrow}\quad
\frac{t}{1-t},
$$
where the limit is to be interpreted as $\infty$ in the case where $t=1$
(and where the sum  on the left-hand side 
is to be interpreted as $\infty$ whenever $\lambda_p = 0$).
Moreover, we also have $\lambda_1/n \to (1+\sqrt{t})^2$ and $\lambda_p/n \to (1-\sqrt{t})^2$ almost surely.
\end{appxlem}

\begin{proof}
Consider first the case where $t$ satisfies $0 < t < 1$.
Write $\nu_1\geq \nu_2 \geq \dots \geq \nu_p$ for the ordered eigenvalues
of $V'V/n$, and note that $\nu_i = \lambda_i/n$, $1\leq i\leq p$.
Write $F_n$ for the empirical cumulative distribution function
(c.d.f.)\ 
of the $\nu_i$'s, i.e., for $x\in\R$,
$F_n(x)$ is the fraction of eigenvalues of $V'V/n$ that do not exceed $x$,
and note that $F_n$ is a random c.d.f.\ as it depends on $V'V/n$.
Under the maintained assumptions,
$F_n$ converges weakly to a non-random limit
$F$, except on a set of probability zero; cf. Theorem 3.6 
of~\cite{baisilverstein10}.
The limit c.d.f.\ $F$ corresponds to the so-called\sout{ Marchenko-Pastur} \textcolor{red}{Marchenko--Pastur} 
distribution, which is supported by the interval $[a,b]$ with 
$a = (1-\sqrt{t})^2$ and $b=(1+\sqrt{t})^2$, and which has 
a density given by $f(x) = \sqrt{(x-a)(b-x)}/( 2 \pi t x)$
for $a\leq x \leq b$.
For any bounded continuous function $g(\cdot)$, it follows 
that $\int g(x) F_n( \text{d} x) \to \int g(x) f(x) \text{d} x$ almost
surely in view of the Portmanteau Theorem.
And under the maintained assumptions, we also see that the smallest
and the largest
eigenvalue of $V'V/n$, i.e.,  $\nu_p = \lambda_p/n$ and
$\nu_1 = \lambda_1/n$, 
converge almost surely to $a$
and to $b$, respectively, whenever $0<t \leq 1$;
see Theorem~5.11 of~\cite{baisilverstein10}.
Recall that $a=(1-\sqrt{t})^2$ is positive, 
and define a function $g(\cdot)$ as
$g(x) = 1/x$ if $x\geq a/2$ and set $g(x) = 2/a$ otherwise.
Now first note that
$\int g(x) F_n( \text{d} x) \to \int g(x) f(x) \text{d} x$, except on a  probability
zero event, because $g(\cdot)$ is continuous and bounded.
Second,\sout{ note that we have} \textcolor{red}{we have that} $\int g(x) f(x) \text{d} x = \int (1/x) f(x) \text{d} x$ because 
$g(x) = 1/x$ on the support of $f(\cdot)$.
And third, observe that
$\int g(x) F_n( \text{d} x)$ and $\int (1/x) F_n( \text{d} x)$ converge to the same
limit, except on a probability zero event, because
$\nu_p$
converges to $a$ almost surely.
These three observations\sout{ entail} \textcolor{red}{imply} that
$$
\int \frac{1}{x} F_n( \text{d} x) \quad\stackrel{a.s.}{\longrightarrow }\quad
\int \frac{1}{x} f(x) \text{d} x.
$$
But the left-hand side in the preceding display can also be written
as $\frac{1}{p} \sum_{i=1}^p \frac{1}{\nu_i} = 
\frac{n}{p} \sum_{i=1}^p \frac{1}{\lambda_i}$; and the
right-hand side is equal to  $1/(1-t)$ in view of Lemma~\ref{mp1}.
It follows that $\sum_{i=1}^p 1/\lambda_i$ converges to
$ t/(1-t)$ almost surely, as required.

In the case where $t=0$, choose an arbitrary number
$\tilde{t}$ with  $0 < \tilde{t} < 1$, and choose numbers  
$\tilde{p} \geq p$ that go to infinity with $n$
such that $\tilde{p}/n \to \tilde{t}$ as $n \to\infty$.
Write $\tilde{V}$ for the $n\times \tilde{p}$ matrix
that forms the upper left block of the double array
$(V_{i,j})_{i\geq 1, j\geq 1}$,
and denote the eigenvalues of $\tilde{V}'\tilde{V}$ by
$\tilde{\lambda}_1 \geq \dots \geq \tilde{\lambda}_{\tilde{p}}$.
Because $V$ is a sub-matrix of $\tilde{V}$, Cauchy's interlacing theorem
entails that 
$\sum_{i=1}^p 1/\lambda_i \leq \sum_{i=1}^{\tilde{p}} 1/\tilde{\lambda}_i$,
and also 
that $\tilde{\lambda}_{\tilde{p}} \leq \lambda_p \leq \lambda_1
\leq \tilde{\lambda}_1$.
And from the case considered in the preceding paragraph, it follows that
$\sum_{i=1}^{\tilde{p}} 1/\tilde{\lambda}_i \to \tilde{t}/(1-\tilde{t})$,
that $\tilde{\lambda}_{\tilde{p}}/n \to (1-\sqrt{\tilde{t}})^2$,
and that $\tilde{\lambda}_1/n \to (1+\sqrt{\tilde{t}})^2$,
almost surely.
Taken together, we obtain that
$\limsup \sum_{i=1}^p 1/\lambda_i \leq \tilde{t}/(1-\tilde{t})$,
that $(1-\sqrt{\tilde{t}})^2 \leq \liminf \lambda_p/n$,
and that $\limsup \lambda_1/n \leq (1+\sqrt{\tilde{t}})^2$,
almost surely.
Letting $\tilde{t} \downarrow 0$  gives the desired result.

In the case where $t=1$, we already know that $\lambda_1/n \to 4$ and
$\lambda_p/n \to 0$ almost surely. Choose an arbitrary number 
$\tilde{t}$ with $0 < \tilde{t} < 1$, and let $\tilde{p} \leq p$
be such that $\tilde{p}/n \to \tilde{t}$.
Now repeat the argument in the preceding paragraph, with the role of
$p$ and $\tilde{p}$ exchanged, to conclude that
$\tilde{t}/(1-\tilde{t}) \leq \liminf \sum_{i=1}^p 1/\lambda_i$ almost surely. The result follows by letting $\tilde{t} \uparrow 1$.
\end{proof}

\begin{remark*}\normalfont
The lemma continues to hold if $\sum_{i=1}^p 1/\lambda_i$
is replaced by $\sum_{i: \lambda_i > 0} 1/\lambda_i$.
And the lemma can be adapted to cover also the case where 
$t \in (1,\infty]$.
\end{remark*}

Throughout the following, we use symbols like $\chi^2_k(\lambda)$
to denote a random variable that is chi-square distributed
with $k\geq 1$ degrees of freedom and non-centrality parameter $\lambda\geq 0$.
Note that $\chi^2_k(\lambda)$ has the same distribution as
the sum  of $k + 2 J_\lambda$ i.i.d.\ central chi-square distributed random
variables with one degree of freedom that are also independent of $J_\lambda$,
where $J_\lambda$ is Poisson with mean $\lambda/2$. In that sense,
the law of  $\chi^2_k(\lambda)$ can be viewed as a central
chi-square distribution with random degrees of freedom equal to 
$k + 2 J_\lambda$, and we will also denote this distribution by
$\chi^2_{k+2 J_\lambda}$. 
From this, it follows that $\chi^2_k(\lambda)$ is
stochastically larger than $\chi^2_k(0)$, whence
$\E[ (1/ \chi^2_k(\lambda))^m] \leq \E[(1/\chi^2_k(0))^m]$
for each $m\geq 1$ (where the expected values can be infinite). 
Also note that\sout{ $\E[(1/\chi^2_k(0))^m] = 1/((k-2)\cdots (k- 2 m))$}
\textcolor{red}{$\E[(1/\chi^2_k(0))^m] = \prod_{i=1}^m [1/(k-2i)]$} 
whenever $k$ and $m$ are positive integers satisfying $k > 2 m$.

\begin{appxlem} \label{chisq}
Fix an integer $m\geq 1$, and consider random variables $\chi^2_k(\lambda)$
with $k > 2 m$. Then
$$
\E\left[  \left( \frac{ k + \lambda}{ \chi^2_k(\lambda)}
\right)^m
\right]
\quad \longrightarrow \quad 1
$$
as $k+\lambda \to \infty$.
\end{appxlem}

\begin{proof}
Fix $m\geq 1$.  The mean and the variance of $\chi^2_k(\lambda)$ are given
by $k + \lambda$ and $2(k + 2\lambda)$, respectively.
Using this fact and Chebyshev's inequality, we see that
$\chi^2_k(\lambda)/(k+\lambda)\to 1 $, and also that\sout{ $( (k+\lambda) / \chi^2_k(\lambda) )^{m} \to 1$}
\textcolor{red}{$[ (k+\lambda) / \chi^2_k(\lambda) ]^{m} \to 1$},
in probability as $k+\lambda\to\infty$. It remains to show, for 
integers $k_a$ and reals $\lambda_a$,  $a\geq 1$,
satisfying $k_a > 2 m$, $\lambda_a\geq 0$,  
and $k_a + \lambda_a \to \infty$ as $a\to \infty$,
that the random variables\sout{ $( (k_a+\lambda_a) / \chi^2_{k_a}(\lambda_a) )^{m}$} \textcolor{red}{$[ (k_a+\lambda_a) / \chi^2_{k_a}(\lambda_a) ]^{m}$}
are uniformly integrable. In other words, for each fixed $\epsilon>0$,
we need to find a constant $M$ so that
\begin{equation}
\label{chisq.1}
\E\left[
	\left(
	\frac{k_a+\lambda_a}{\chi^2_{k_a}(\lambda_a)}
	\right)^m
	\left\{
	\left(
	\frac{k_a+\lambda_a}{\chi^2_{k_a}(\lambda_a)}
	\right)^m
	\,>\,M
	\right\}
\right]\quad < \quad\epsilon
\end{equation}
holds for each $a$,
where we use symbols like $\{ A \}$ to denote the indicator 
function of the event $A$. Set\sout{ $U_a = ((k_a + \lambda_a)/\chi^2_{k_a}(\lambda_a))^m$} \textcolor{red}{$U_a = [(k_a + \lambda_a)/\chi^2_{k_a}(\lambda_a) ]^m$},
note that the $U_a$'s are all integrable because $k_a > 2 m$,
and recall that $\chi^2_{k_a}(\lambda_a)$ is distributed
as $\chi^2_{k_a + 2 J_a}(0)$ for an independent Poisson
random  variable $J_a$ with mean $\lambda_a/2$.
For later use, we also recall the Chernoff bound
$\P(J_a < \lambda_a/4) \leq (e/2)^{-\lambda_a/4}$.
To derive \eqref{chisq.1}, i.e., to show uniform integrability of the
$U_a$'s, we may assume that either $k_a \to\infty$ or $\lambda_a \to\infty$ 
as $a\to\infty$, by switching to subsequences if necessary
(because $k_a+\lambda_a\to\infty$).

Assume that $k_a\to\infty$. Here, we first consider the (finitely many)
$a$'s for which  $k_a \leq 4 m$. 
For these, we can find a constant $M_1$ so that \eqref{chisq.1} holds
whenever $M\geq M_1$ because the $U_a$'s are integrable.
It remains to consider the $a$'s with $k_a > 4 m$.
For these $a$'s, we derive uniform integrability 
by showing that the second moments of the $U_a$'s are bounded.\sout{ Indeed,}
\textcolor{red}{Using the law of iterated expectation}, $\E[ U_a^2]$ is given by the sum of
\begin{align}
&\E\left[
\prod_{i=1}^{2 m} 
\left(
\frac{ k_a+\lambda_a}{k_a-2 i+ 2 J_{\lambda_a}}
\right)
\left\{ J_{\lambda_a}  \geq  \lambda_a/4 \right\}
\right] \label{chisq.2} \qquad \xout{\text{and}} \\
\intertext{\textcolor{red}{and}}
&\E\left[
\prod_{i=1}^{2 m} 
\left(
\frac{ k_a+\lambda_a}{k_a-2 i+ 2 J_{\lambda_a}}
\right)
\left\{ J_{\lambda_a}  <  \lambda_a/4 \right\}
\right]. \label{chisq.3}
\end{align}
The expression in \eqref{chisq.2} is bounded by
$(4 m + 3)^{2 m}$ because  the $i$-th fraction in the integrand of
\eqref{chisq.2} is bounded by
$(k_a+\lambda_a)/(k_a-2 i + \lambda_a/2) \leq 
 (k_a+\lambda_a)/(k_a-4 m + \lambda_a/2) \leq k_a/(k_a-4m) + 2 \leq
 4 m + 1 + 2 $
(the first inequality holds in view of $J_{\lambda_a} \geq \lambda_a/4$,
and the last inequality holds because $k_a > 4 m$).
To bound the expression in \eqref{chisq.3},\sout{ we first} note that
the $i$-th  fraction in the integrand in \eqref{chisq.3} is bounded
by $k_a / (k_a - 2 i) + \lambda_a \leq 4m+1 + \lambda_a$.
Using this and the Chernoff's Poisson tail bound mentioned earlier,
it follows that \eqref{chisq.3} is bounded by
$(4m+1+\lambda_a)^{2 m} (e/2)^{-\lambda_a/4}$.
This upper bound is a continuous function of $\lambda_a\geq 0$ that
converges to zero as $\lambda_a\to\infty$. Hence, the upper bound
is itself bounded by a constant, uniformly in $\lambda_a\geq 0$,
as required.

Now assume that $\lambda_a \to\infty$. 
We decompose the expression on the left-hand side of \eqref{chisq.1}\sout{,
i.e., $\E[ U_a \{U_a > M\}]$,} into the sum of
$\E[ U_a \{U_a > M\}\{ J_{\lambda_a} < \lambda_a/4\}]$
and 
$\E[ U_a \{U_a > M\}\{ J_{\lambda_a}\geq \lambda_a/4\}]$, and find
$M$ so that each of these two terms is bounded by $\epsilon/2$.
To this end, note that $\chi^2_{k_a}(\lambda_a)$ is 
stochastically larger than $\chi^2_{k_a}(0)$, so that
$\E[ U_a \{U_a > M\}\{ J_{\lambda_a} < \lambda_a/4\}]$ can be bounded by
\begin{align*}
\E\left[
\left(
	\frac{ k_a + \lambda_a}{ \chi^2_{k_a}(0)}
\right)^m
\right]
\; \P
\left(
	J_{\lambda_a} < \lambda_a/4
\right)
\quad\leq\quad (2 m+1+\lambda_a)^m (e/2)^{-\lambda_a/4}.
\end{align*}
In the preceding display,
the inequality is derived by writing the expected value
as $\prod_{i=1}^m (k_a + \lambda_a)/(k_a - 2 i)$, 
by noting that $k_a/(k_a-2 i) \leq 2 m + 1$ because $k_a > 2 m$,
and upon using Chernoff's tail bound for the Poisson.
Since $\lambda_a \to\infty$ here, the upper bound in the preceding display
is smaller than $\epsilon/2$ for sufficiently large  $\lambda_a$'s, e.g.,
$\lambda_a > \lambda_\ast$. And for the (finitely many) $a$'s for which
$\lambda_a \leq \lambda_\ast$, we can find a constant $M_2$
so that 
$\E[ U_a \{U_a > M\}\{ J_{\lambda_a} < \lambda_a/4\}]$ is less than
$\epsilon/2$ whenever $M\geq M_2$.
Lastly, 
$\E[ U_a \{U_a > M\}\{ J_{\lambda_a} \geq  \lambda_a/4\}]$ is bounded by
\begin{align*}
&
\E\left[
\left( \frac{k_a + \lambda_a}{ \chi^2_{k_a + \lceil \lambda_a/2\rceil}(0)}
\right)^m
\left\{
\left( \frac{k_a + \lambda_a}{ \chi^2_{k_a + \lceil \lambda_a/2\rceil}(0)}
\right)^m
>M
\right\}
\right]
\\
&
\leq \quad
2^m \E\left[
\left( \frac{k_a + \lceil \lambda_a/2\rceil }{ 
	\chi^2_{k_a + \lceil \lambda_a/2\rceil}(0)}
\right)^m
\left\{
\left( \frac{k_a + \lceil \lambda_a/2 \rceil}{ 
	\chi^2_{k_a + \lceil \lambda_a/2\rceil}(0)}
\right)^m
>2^{-m} M
\right\}
\right],
\end{align*}
where $\lceil x \rceil$ denotes the smallest integer not smaller than $x$,
and where the inequality is based on the observation that
$k_a + \lambda_a = 2 ( k_a/2 + \lambda_a/2) < 2 
(k_a + \lceil \lambda_a/2\rceil)$. Set 
$\tilde{k}_a = k_a + \lceil \lambda_a/2\rceil$ and set
$\tilde{\lambda}_a = 0$. To find an $M_3\geq M_2$
such that the expression on the right-hand side of the preceding display
is less than $\epsilon/2$ whenever $M\geq M_3$, it suffices to show that
the random variables\sout{ $( (\tilde{k}_a +\tilde{\lambda}_a) /  
\chi^2_{\tilde{k}_a}(\tilde{\lambda}_a) )^m$}
\textcolor{red}{$[ (\tilde{k}_a +\tilde{\lambda}_a) /  
\chi^2_{\tilde{k}_a}(\tilde{\lambda}_a) ]^m$}
are uniformly integrable. Since $\tilde{k}_a \to \infty$ as $a\to\infty$,
this has already been established  in the second paragraph of the proof.
\end{proof}

\begin{appxlem}
\label{quadrat1}
Consider the $n\times p$ matrix $V$ as in Section~\ref{s3}, and
let $w$ be a unit-vector in $\R^p$ ($n\geq p \geq 3$).  
Then $\E[ wV'V w/n] = 1$ and $\Var[ w'V'Vw/n] \to 0$ as $n\to\infty$,
irrespective of the behavior of $p$; in particular, we have
$w' V'V w/n \to 1$ in probability.
\end{appxlem}

\begin{proof}
Setting $W^{(n)}_i = (\sum_{j=1}^p V_{i,j} w_j)^2$, we see that
$w'V'Vw/n$ can be written as $(1/n) \sum_{i=1}^n W^{(n)}_i$, i.e.,
as the average of $n$ i.i.d.\ random variables which have mean 
$\sum_{j=1}^p w_j^2 = 1$. And the variance of $W^{(n)}_i$ is given by
\begin{align*}
&\sum_{a,b,c,d=1}^p 
w_a w_b w_c w_d
\E[V_{i,a} V_{i,b} V_{i,c} V_{i,d}]
- 1
\;\;=\;\; \E[V_{1,1}^4]\sum_{a=1}^p  w_a^4 + 
	3  \sum_{\stackrel{a,b=1}{a\neq b}}^p w_a^2w_b^2 - 1
\\
&\;\;=\;\; (\E[V_{1,1}^4] -1) \sum_{a=1}^p w_a^4 + 
	2 \sum_{\stackrel{a,b=1}{a\neq b}}^p w_a^2 w_b^2
\;\;\leq \;\; \E[V_{1,1}^4]+1.
\end{align*}
In the preceding display, the first equality is based  on the fact
that $\E[V_{i,a} V_{i,b}, V_{i,c}, V_{i,d}] = 0$ whenever one
of the indices  $a,b,c,d$ differs from the others, while the
second equality and the inequality are derived from the facts
that $\sum_{j=1}^p w_j^2 = 1$ and that $1\leq \E[V_{1,1}^4]$.
We hence can bound $\Var[ w'V'Vw/n]$ by  $(\E[V_{1,1}^4]+1)/n$.
\end{proof}

\begin{appxlem}
\label{quadrat2}
Assume that Theorem~\ref{t1} applies.
If $t+\delta^2 > 0$, then
$$ \xout{
\frac{p}{p + \beta'X'X\beta}\quad\longrightarrow \quad \frac{t}{t+\delta^2}
\qquad\text{ and }\qquad
\beta'\Sigma\beta \frac{p}{p + \beta'X'X\beta} 
\quad \longrightarrow \quad \frac{t \delta^2}{t + \delta^2}
}
$$
$$
\textcolor{red}{\frac{p}{p + \beta'X'X\beta}\quad\longrightarrow \quad \frac{t}{t+\delta^2}, \qquad
\beta'\Sigma\beta \frac{p}{p + \beta'X'X\beta} 
\quad \longrightarrow \quad \frac{t \delta^2}{t + \delta^2},}
$$
in probability as $n\to\infty$; in case $\delta^2=\infty$, the two limits
are to be interpreted as $0$ and as $t$, respectively.
These statements continue to hold if $p$ is replaced by $p+k$ in the
numerators of the preceding display, for some fixed $k\in \N$.
\end{appxlem}

\begin{proof}
Set $w = \Sigma^{1/2} \beta/(\beta'\Sigma\beta)^{1/2}$, 
note that $||w||=1$, and that
$\beta'X'X\beta/n$ can be written as 
$( \beta'\Sigma\beta)\,  w'V'V w /n$. Using Lemma~\ref{quadrat1},
it follows that $\beta'X'X\beta/n \to \delta^2$ in probability.
Because $\beta'X'X\beta/n$ and $\beta'\Sigma\beta$ both converge
to $\delta^2$ (the former in probability), and because $p/n\to t$,
the continuous mapping theorem gives both limits
in case $\delta^2<\infty$, and the first limit in case $\delta^2=\infty$.
In the remaining case where $\delta^2=\infty$, write
the quantity of interest as\sout{ $(p/n) / ( (p/n) / \beta'\Sigma\beta + w'V'Vw / n)$}
\textcolor{red}{$(p/n) / [ (p/n) / \beta'\Sigma\beta + w'V'Vw / n]$},
which is easily seen to converge to $t$, as claimed.
The last statement is trivial because $p/n$ and $(p+k)/n$ converge
to the same limit.
\end{proof}

\begin{proof}[\bf Proof of Theorem~\ref{t1}]
For the maximum-likelihood estimator, we have
$\rho_2(\hat{\beta}_{ML},\beta,X) = \trace(\Sigma(X'X)^{-1})
= \trace((V'V)^{-1})$ almost surely,
so that $\rho_2(\hat{\beta}_{ML},\beta,X) 
\to t/(1-t)$ in probability by Lemma~\ref{mp2}.
For $\hat{\beta}(c)$, 
note that $\rho_2(\hat{\beta}{(c)}, \beta,X)$ is given by
\eqref{eq:oospe.js} almost surely (where the expected values are to be
understood as conditional on $X$, \textcolor{red}{denoted by $\E[ \cdot \Vert X]$ in the following)};
using Proposition~\ref{propa2} \textcolor{red}{conditional on $X$}, we can thus write
$\rho_2(\hat{\beta}{(c)}, \beta,X)$ also as
\begin{align}
\label{pt1.1}
\begin{split}
&\trace((V'V)^{-1})   \left( 1 
- 2 c 
  \E\left[ \left.\frac{ p  }{\chi^2_p(\beta'X'X\beta)}
  	\right\| X \right]
+ (c^2 + 4 c /p)  
  \E\left[ \left. \frac{p^2}{(\chi^2_{p+2}(\beta'X'X\beta))^2} 
  \right\|X  \right]
  \right)
\\ & \quad
+ (c^2 + 4 c /p) 
  \E\left[ \left. \frac{p^2\beta'\Sigma\beta}{(\chi^2_{p+4}(\beta'X'X\beta))^2} 
  \right\|X  \right]
\end{split}
\end{align}
almost surely\sout{ (note that Proposition~\ref{propa2} assumes that $X$ is non-random,
so we use this result here conditional on $X$)}. 
Our first goal is to show that\sout{ $\sup_{0\leq c\leq C} | \rho_2(\hat{\beta}{(c)},\beta,X) - r(\delta^2,c,t)|$
converges to zero in probability.}\textcolor{red}{
$\sup_{0\leq c\leq C} | \rho_2(\hat{\beta}{(c)},\beta,X) - r(\delta^2,c,t)|=o_P(1)$.}

In the case where $t+\delta^2 = 0$, note that $r(0, c, 0) = 0$,
so that 
$\sup_{0\leq c\leq C} | \rho_2(\hat{\beta}{(c)},\beta,X) - r(\delta^2,c,t)|$
is bounded from above by
$$
\trace((V'V)^{-1}) \left(
	1 + 2 C \E\left[\frac{p}{\chi^2_p(0)}\right]
	  + (C^2+4C/p) \E\left[ \frac{p^2}{(\chi^2_{p+2}(0))^2}\right]\right)
	  + (C^2+4C/p) \beta'\Sigma\beta 
	  	\E\left[ \frac{p^2}{(\chi^2_{p+4}(0))^2}\right]
$$
because $\chi^2_{p+k}(\beta'X'X \beta)$ is stochastically
larger than $\chi^2_p(0)$ for each $k\geq 0$.\sout{
Moreover, for each $p\geq 3$, we have
$\E[p/\chi^2_p(0)] = p/(p-2) \leq 3$
and $\E[p^2/(\chi^2_{p+4}(0))^2] 
\leq \E[p^2/(\chi^2_{p+2}(0))^2] 
= p^2/(p(p-2)) \leq 3$.} Because both $\trace((V'V)^{-1})$ and $\beta'\Sigma\beta$
converge to zero (the former in probability) \textcolor{red}{and the expected values are bounded by 3 for each $p \geq 3$}, it follows that
the expression in the preceding display converges to\sout{ 0} \textcolor{red}{zero} in probability.

In the case where $t+\delta^2>0$, we\sout{ first} use the formula for
$\rho_2(\hat{\beta}{(c)},\beta,X)$ in \eqref{pt1.1},
the formula for $r(\delta^2,c,t)$, and the triangle inequality to bound
$\sup_{0\leq c\leq C} | \rho_2(\hat{\beta}{(c)},\beta,X) - r(\delta^2,c,t)|$
from above by
\begin{align}
\nonumber
\begin{split}
&\Bigg| \trace((V'V)^{-1} - \frac{t}{1-t}\Bigg|
\;+\;
2 C \Bigg| \trace((V'V)^{-1} 
  \E\left[ \left.\frac{ p  }{\chi^2_p(\beta'X'X\beta)} \right\| X \right]
  - \frac{t^2}{(1-t)(t+\delta^2)} 
\Bigg|
\\ &
\;+\; C^2 \Bigg|
	\trace((V'V)^{-1} 
  \E\left[ \left. \frac{p^2}{(\chi^2_{p+2}(\beta'X'X\beta))^2} 
  \right\|X  \right]
  - \frac{t^3}{(1-t)(t+\delta^2)^2} 
\Bigg|
\\&
\;+\;
C^2 \Bigg|
  \E\left[ \left. \frac{p^2\beta'\Sigma\beta}{(\chi^2_{p+4}(\beta'X'X\beta))^2} 
  \right\|X  \right]
  - \frac{t^2\delta^2}{(t+\delta^2)^2}
\Bigg|
\\&
\;+\; 4 (C/p)  \Bigg(
  \trace((V'V)^{-1}) 
  \E\left[ \left. \frac{p^2}{(\chi^2_{p+2}(\beta'X'X\beta))^2} 
  \right\|X  \right]
  +
  \beta'\Sigma\beta
  \E\left[ \left. \frac{p^2}{(\chi^2_{p+4}(\beta'X'X\beta))^2} 
  \right\|X  \right]
\Bigg).
\end{split}
\end{align}
\sout{The expression in the preceding display is the sum of five terms,
where we already know that the first term converges to zero in probability.}
\textcolor{red}{We already know that the first term of the preceding display converges to zero in probability.}
For the remaining four terms, we note that
either $p\to\infty$ (in case $t>0$),
or $\beta'X'X\beta \to \infty$
in probability (in case $\delta^2 > 0$ in view of Lemma~\ref{quadrat1} 
because $n \beta'X'X\beta/n$ can  be written as
$n \beta'\Sigma\beta (w'V'Vw/n)$, where the unit-vector $w$
is given by $\Sigma^{1/2}\beta/ (\beta'\Sigma\beta)^{1/2}$).
Using the assumption that $p\geq 3$,\sout{ it is easy to see that}
Lemma~\ref{chisq} can be used to deal with the expected values
in the preceding display.
Together with the fact that $\trace((V'V)^{-1}) \to t/(1-t)$ in probability, 
this shows
that the sum of the four remaining terms converges to the same limit as
\begin{align}
\nonumber
\begin{split}
& 2 C \frac{t}{1-t} \Bigg| 
	\frac{p}{p+\beta'X'X\beta}
  - \frac{t}{t+\delta^2} 
\Bigg|
\;+\; C^2 \frac{t}{1-t} \Bigg|
  \frac{p^2}{(p+2+\beta'X'X\beta)^2} 
  - \frac{t^2}{(t+\delta^2)^2} 
\Bigg|
\\&
\;+\;
C^2 \Bigg|
  \beta'\Sigma\beta\frac{p^2}{(p+4+\beta'X'X\beta)^2} 
  - \frac{t^2\delta^2}{(t+\delta^2)^2}
\Bigg|
\\&
\;+\; 4 (C/p)  \Bigg(\frac{t}{1-t}
  \frac{p^2}{(p+2+\beta'X'X\beta)^2} 
  +
  \beta'\Sigma\beta \frac{p^2}{(p+4+\beta'X'X\beta)^2} 
\Bigg).
\end{split}
\end{align}
Lemma~\ref{quadrat2} now entails that
the expressions in the preceding display converge to zero in probability.
(For the last term in the preceding display, we either have $p\to\infty$
if $t>0$, or $t=0$.)

Our second goal is to show that $\sup_{0\leq c \leq C} 
|r(\hat{\delta}^2,c,p/n) - r(\delta^2,c,t) |$ also 
converges to zero in probability. \sout{To this end, we} \textcolor{red}{We} can bound the
supremum in question by
\begin{align*}
& \Bigg| 	\frac{p/n}{1-p/n} - \frac{t}{1-t}
\Bigg|
\;+\;
2 C \Bigg|
	\frac{(p/n)^2}{(1-p/n)(p/n + \hat{\delta}^2)} -
	\frac{t^2}{(1-t)(t+\delta^2)}
\Bigg|
\\
&\;+\; C^2  \Bigg|
	\frac{(p/n)^3}{(1-p/n)(p/n+\hat{\delta}^2)^2} -
	\frac{t^3}{(1-t)(t+\delta^2)^2}
\Bigg|
\;+\; C^2 \Bigg|
	\hat{\delta}^2 \frac{ (p/n)^2}{(p/n+\hat{\delta}^2)^2} -
	\delta^2 \frac{t^2}{(t+\delta^2)^2}
\Bigg|.
\end{align*}
The expression in the preceding display is the sum of four terms,
where each term is of the form $| f(p/n, \hat{\delta}^2) - f(t,\delta^2)|$
for some function $f(\cdot,\cdot)$ which is
continuous on $[0,1)\times [0,\infty]$.
Since $p/n\to t$ by assumption and $\hat{\delta}^2 \to\delta^2$ in
probability (as is easy to see), the continuous mapping theorem entails
that each of the four terms in the preceding display converges to zero
in probability.
\end{proof}

\begin{proof}[\bf Proof of Theorem~\ref{t2}]
On the almost-sure event where\sout{ $X'X>0$} \textcolor{red}{$X'X$ is invertible}, 
$\rho_2(\hat{\beta}{(c)},\beta,X)$ is given by
the formula \eqref{pt1.1}. For\sout{ $X'X>0$} \textcolor{red}{invertible $X'X$} and
for $\beta$ such
that $\beta' (X'X/n)\beta = d^2$, 
the expression in  \eqref{pt1.1} depends on $\beta$ only 
through $\beta'\Sigma\beta$ which is at most 
$d^2 / \lambda_p(V'V/n)$, where 
$\lambda_p(\dots)$ denotes the smallest eigenvalue of the indicated
matrix.
[Indeed, we have $d^2 = \| (V'V/n)^{1/2} \Sigma^{1/2}\beta\|^2$
and $\beta'\Sigma\beta = \beta' \Sigma^{1/2} (V'V/n)^{1/2} 
\,[ (V'V/n)^{-1}]\, (V'V/n)^{1/2} \Sigma^{1/2}\beta \leq d^2 / 
\lambda_p(V'V/n)$.]
It follows that $\sup_{\beta\in \R^p}
\rho_2(\hat{\beta}{(c)},\beta,X)$ can almost surely be written as
the supremum of
\begin{align}
\nonumber 
\begin{split}
& \trace((V'V)^{-1})   \left( 1 
- 2 c 
  \E\left[ \frac{ p  }{\chi^2_p(n d^2)}  \right]
+ (c^2 + 4 c /p)  
  \E\left[ \frac{p^2}{(\chi^2_{p+2}(n d^2))^2} 
    \right]
  \right)
+ \frac{c^2 + 4 c /p}{ \lambda_p(V'V/n)}
  \E\left[ \frac{p^2 d^2  }{(\chi^2_{p+4}(n d^2))^2} 
    \right] 
\end{split}
\end{align}
over $d^2\geq 0$ or, equivalently, over $d\geq 0$.
Write $R_\ast(d^2,c,n,p)$ for the random variable in the preceding display.
The claim will follow if we can show that
$$
\sup_{0\leq c\leq C} \sup_{d\geq 0} \left|
	R_\ast(d^2,c,n,p) - R(d^2,c,t)
\right|
$$
converges to zero in probability.
Now use the triangle inequality to bound the expression in the preceding
display by the supremum of
\begin{align*}
\begin{split}
& \Bigg| 
	\trace((V'V)^{-1})  - \frac{t}{1-t} \Bigg|
+ 2 C \Bigg|
	\trace((V'V)^{-1})
	\E\left[ \frac{p}{\chi^2_p(n d^2)}\right] - \frac{t^2}{(1-t)(t+d^2)}
\Bigg|
\\&
+ C^2 \Bigg| 
	\trace((V'V)^{-1}) 
	\E\left[ \frac{p^2}{(\chi^2_{p+2}(n d^2))^2}\right]  
	- \frac{t^3}{(1-t)(t+d^2)^2}
\Bigg|
\\ &
+ C^2 \Bigg| \frac{1}{\lambda_p(V'V/n)}
	\E\left[ 
	\frac{p^2 d^2 }{(\chi^2_{p+4}(n d^2))^2}
	\right]
	- \frac{ t^2 d^2}{(1-\sqrt{t})^2 (t+d^2)^2}
\Bigg|
\\&
+ 4C 
	\trace((V'V)^{-1})
	\E\left[ \frac{p}{(\chi^2_{p+2}(n d^2))^2}\right]  
+ 4 C
	\frac{1}{\lambda_p(V'V/n)}
	\E\left[ \frac{p d^2 }{(\chi^2_{p+4}(n d^2))^2}
	\right]
\end{split}
\end{align*}
over $d \geq 0$.
\sout{The expression in the preceding display is the sum of six terms,
where the first one obviously converges to zero in probability,
and where each of the remaining five
is of the form $|U_n f_n(d) - u f(d)|$ multiplied
by a constant, where $U_n$ is a random variable
that converges to $u \in \R$ in probability (cf. Lemma~\ref{mp2}), and where
$f_n(d)$  and $f(d)$ are  functions of $d \in [0,\infty)$.} \textcolor{red}{The first term of the preceding display converges to zero in probability and each of the remaining five terms
is of the form $|U_n f_n(d) - u f(d)|$ multiplied
by a constant, where $U_n$ is a random variable
that converges to $u \in \R$ in probability (cf. Lemma~\ref{mp2}), and where
$f_n(d)$  and $f(d)$ are  functions of $d \in [0,\infty)$.}
\sout{[The two possible cases for $U_n$ and $u$, respectively, are
$U_n = \trace((V'V)^{-1})$ and $u=t/(1-t)\geq 0$, as well as
$U_n = 1/\lambda_p(V'V/n)$ and $u = 1/(1-\sqrt{t})^2>0$.]}
For each of the remaining five terms, we now proceed as follows:
To show that $\sup_{d\geq 0} |U_n f_n(d) - u f(d)| \to 0$ in
probability, it suffices to show 
that $|U_n| \sup_{d\geq 0} |f_n(d) - f(d)| \to 0$ in probability
(because $|U_n f_n(d) - u f(d)|$ is bounded by
$|U_n| |f_n(d) - f(d)| + |f(d)| | U_n - u| \leq
|U_n| |f_n(d) - f(d)| + | U_n - u|$, where the inequality follows
upon noting that $f(d) \leq 1$\sout{, which is easily verified for
each of the five remaining terms in the preceding display}).
Let $d_n$, $n\geq 1$, be a sequence of maximizers (or near maximizers)
of  $|f_n - f|$\sout{ , e.g., so that
$\sup_{d\geq 0}|f_n(d)-f(d)| \leq |f_n(d_n)-f(d_n)|+1/n$}.
Because $[0,\infty]$ is compact, we may assume that
$d_n \to \delta \in [0,\infty]$ 
(replacing the original sequence by a convergent
subsequence if necessary).
Moreover, it is easy to see that $\lim_{d\to\infty}f(d) = 0$
(for each of the five remaining terms  in the preceding display).
In particular, we can extend $f(\cdot)$ to
a continuous function on $[0,\infty]$ by setting $f(\infty)=0$.
To complete the proof, we now show that either
(a) $f_n(d_n) - f(\delta) \to 0$ as $n\to\infty$,
or (b) $u =0$ and $|f_n(d_n)|$ is bounded.
In the case where $t+\delta^2 > 0$, we see that either
$p\to\infty$ or $n d_n \to\infty$, and it is not difficult to
conclude  that case (a) occurs by using Lemma~\ref{chisq}
and the assumption that $p\geq 3$.
And in the case where $t+\delta^2 = 0$, we see that
case (b) occurs  for the first, second, and fourth of the
five remaining terms in the preceding display (by arguing as in
the corresponding part in the proof of Theorem~\ref{t1}, mutatis
mutandis), while
again case (a) occurs for the third and the fifth
as the expected values are bounded and $d_n$ tends to zero.
\end{proof}

\begin{proof}[\bf Proof of Theorem~\ref{t3}]
It suffices to show the statements for $\hat{\beta}(c_n)$.
Set $R_\ast = \sup_{\delta^2\geq 0} R(\delta^2,c,t)$
for $R(\cdot,\cdot, \cdot)$ as in Theorem~\ref{t2}, and note that
$\sup_{\beta\in \R^p} \rho_2(\hat{\beta}(c_n),\beta,X) -
\rho_2(\hat{\beta}_{ML},\beta,X)$ converges to $R_\ast - t/(1-t)$
in probability by\sout{ theorems} \textcolor{red}{Theorems}~\ref{t1} and~\ref{t2}.
For the first statement, assume either  that
$0< c \leq 2$ and\sout{ $t>((c-2)/(c+2))^2$} \textcolor{red}{$t>[(c-2)/(c+2)]^2$} or that
$c>2$ and $t>0$ hold (the case $c=0$ can not occur there because $t<1$).  
The statement now follows by observing, for such $c$ and $t$, that
$R_\ast - t/(1-t) > 0$, i.e., that
$R(\delta^2,c,t) - t/(1-t) > 0$ for some $\delta^2\geq 0$
(which is elementary but tedious to verify), and by
setting $\epsilon$ equal to, say,
$(R_\ast - t/(1-t))/2$.
For the second statement, assume that $0\leq c\leq 2$ and\sout{ $t \leq ((c-2)/(c+2))^2$}
\textcolor{red}{$t \leq [(c-2)/(c+2)]^2$}, or that $c>2$ and $t=0$, and note,
for such $t$ and $c$, that $R_\ast - t/(1-t) \leq 0$, i.e.,
that $R(\delta^2,c,t) - t/(1-t) \leq 0$ for each $\delta^2\geq 0$.

For the last statement, let $c\in [0,2]$, and take 
$\epsilon>0$. 
We need to show that
$\P( \rho_2(\hat{\beta}(c_n),\beta,X) - 
\rho_2(\hat{\beta}_{ML},\beta,X) > \epsilon)$ converges to zero
for arbitrary sequences of parameters $\beta \in \R^p$ and
$\Sigma$. Because the set $[0,\infty]$ is compact, we may assume that
$\beta'\Sigma\beta$ converges to a limit $\delta^2\in [0,\infty]$
(by considering convergent subsequences if necessary).
It now follows from Theorem~\ref{t1} that
$\rho_2(\hat{\beta}(c_n), \beta,X)$ converges in probability to
the limit $r(\delta^2,c,t)$ as is given in the theorem,
while $\rho_2(\hat{\beta}_{ML}, \beta,X)$ converges in probability to
$t/(1-t)$. The claim now follows from the fact that
$r(\delta^2,c,t) \leq  t/(1-t)$, irrespective of $\delta^2 \in [0,\infty]$
(which, again, is easy if somewhat tedious to verify).
\end{proof}

\end{appendix}

\bibliographystyle{abbrvnat}

\end{document}